\documentclass[12pt,a4paper]{amsart}
\usepackage{amsthm,amssymb,amsmath}
\usepackage[hidelinks]{hyperref}
\usepackage[left=2cm,right=2cm,top=2cm,bottom=2cm]{geometry}
\title{A complex Hadamard matrix of order 94}
\author{Ferenc Sz\"oll\H{o}si}
\date{March 10, 2026, preprint. This research was supported in part by JSPS KAKENHI Grant Number 24K06829}
\dedicatory{Dedicated to Professor Akihiro Munemasa on the occasion of his retirement}
\address{Interdisciplinary Department of Science and Engineering, Shimane University, 1060 Nishikawatsu-cho, Matsue, 690-8504, Shimane, Japan.}
\email{szollosi@riko.shimane-u.ac.jp}
\newtheorem{lemma}{Lemma}
\newtheorem{theorem}{Theorem}

\theoremstyle{definition}

\theoremstyle{remark}
\newtheorem{example}{Example}

\usepackage{booktabs}
\begin{document}
\begin{abstract}
In this paper we modify a fundamental block construction of Kharaghani and Seberry and show how to use certain circulant $\{-1,1\}$-matrices of odd order $p$ to construct a complex Hadamard matrix of order $2p$. In particular, for $p=47$ we use computer-aided methods to discover the necessary circulant matrices, and consequently give a construction of a complex Hadamard matrix of order $94$ for the first time.
\end{abstract}
\maketitle
\section{Introduction and main results}
A real \emph{Hadamard matrix} $H$ of order $n$ is a square $\{-1,1\}$-matrix satisfying $HH^T=nI_n$, where $I_n$ is the identity matrix of order $n$. One of the fundamental constructions of Hadamard matrices come from Williamson's method, who showed how to combine four symmetric circulant $\{-1,1\}$-matrices of odd order $p$ satisfying a certain condition (see equation \eqref{sup} further below) in order to construct a real Hadamard matrix of order $4p$. For long it was conjectured that these so-called \emph{Williamson matrices} always exist, but subsequently for order $p=35$ (and several others, including $47$) they don't \cite{WILKHAR}. The lack of such matrices lead to various generalizations, such as Williamson-type matrices, \emph{good matrices}, \emph{$G$-matrices}, and \emph{best matrices} \cite{DJBIG}. We refer the interested reader to the monograph \cite{SY}.

Let $\mathbf{i}$ denote the complex imaginary unit. In a seminal paper \cite{KS}, Kharaghani and Seberry showed how to use Williamson matrices of order $p$ to construct \emph{complex Hadamard matrices}, i.e., $\{-1,-\mathbf{i},1,\mathbf{i}\}$-matrices with pairwise complex orthogonal rows, of order $2p$. We modify their construction and show how to use circulant matrices when only two of them are symmetric. Furthermore, we exhibit such matrices for $p=47$ for the first time. This leads to the main result of this paper.
\begin{theorem}\label{T1}
There exists a complex Hadamard matrix of order $94$.
\end{theorem}
Prior to this discovery, \DJ okovi\'c \cite{DJgood} constructed good matrices for $p=35$ in $1993$, and consequently a complex Hadamard matrix of order $70$. The next outstanding order seems to be $p=59$, i.e., a complex Hadamard matrix of order $118$, see \cite[Table~11.2]{SY}. We remark that all complex Hadamard matrices are known up to order $18$, see \cite{OS}. For various other generalizations, see e.g., \cite{MUB}, \cite{MW}, and \cite{SCH}.

In Section~\ref{sect2} we outline how to construct from certain circulant $\{-1,1\}$-matrices real and complex Hadamard matrices. We show a modification of a result of Kharaghani and Seberry \cite{KS} in Theorem~\ref{T4}. Then, in Section~\ref{sect3} we describe a computer-aided search for the required circulant matrices, and provide two examples. The combination of Example~\ref{ex1} with Theorem~\ref{T4} proves Theorem~\ref{T1}.
\section{The Goethels--Seidel array and complex Hadamard matrices}\label{sect2}
Throughout this paper we denote by $R$ the back-diagonal matrix of appropriate order. It is easy to see that for any circulant matrix $X$, we have $(XR)=(XR)^T$, hence $RX^T=XR$. A matrix $X$ is \emph{skew}, if $X^T=-X+2I$. The following construction of real Hadamard matrices while similar to Williamson's method, it does not require symmetric blocks. It is called the Goethels--Seidel construction \cite{GSA}.
\begin{theorem}[\cite{GSA}]\label{T2}
Let $A,B,C,D$ be circulant $\{-1,1\}$-matrices of order $p$, such that
\begin{equation}\label{sup}
AA^T+BB^T+CC^T+DD^T=4pI_p.
\end{equation}
Then
\[H=\left[\begin{array}{cccc}
A & BR & CR & DR\\
-BR & A & -D^TR & C^TR\\
-CR & D^TR & A & -B^TR\\
-DR & -C^TR & B^TR & A\end{array}\right]\]
is a real Hadamard matrix of order $4p$.
\end{theorem}
We note that equation \eqref{sup} is invariant up to cyclically shifting, reversing, or multiplying the rows of the matrices by $-1$.

It is well-known that from complex Hadamard matrices of order $p$, a real Hadamard matrix of order $2p$ can be constructed \cite{Turyn}. So it is natural to ask to what extent this can be reversed? In particular, can the blocks of order $p$ within the Goethels--Seidel array be combined in a clever way in order to get a complex Hadamard matrix of order $2p$? An observation by Kharaghani and Seberry \cite{KS} gives a partial answer, see Theorem~\ref{T3} below.

Recall that the square complex matrices $A$ and $B$ of the same order are \emph{amicable} \cite[p.~444]{SY}, if $AB^\ast = BA^\ast$, where $\ast$ denotes the conjugate-transpose. For later convenience, we record the following.
\begin{lemma}\label{L2}
Let $A$ and $B$ be complex circulant matrices of the same order.
\begin{itemize}
\item[(i)] If $A$ and $B$ are self-adjoint, then $A$ and $B$ are amicable.
\item[(ii)] If $A$ and $B$ are real, then $A$ and $BR$ are amicable.
\end{itemize}
\end{lemma}
\begin{lemma}\label{L3}
Let $A$ and $B$ be complex amicable matrices. Let $X:=(A+B)/2+\mathbf{i}(A-B)/2$. Then $XX^\ast = (AA^\ast+BB^\ast)/2$. Furthermore, if both $A$ and $B$ are normal, and $A^\ast$ and $B^\ast$ are amicable, then $X$ is normal.
\end{lemma}
\begin{proof}
We have $XX^\ast=(AA^\ast+BB^\ast)/2+\mathbf{i}(AB^\ast-BA^\ast)/2$, and $X^\ast X=(A^\ast A+B^\ast B)/2-\mathbf{i}(A^\ast B-B^\ast A)/2$. These two expressions agree subject to the noted conditions.
\end{proof}
We note that circulant matrices as well as self-adjoint matrices are normal.
\begin{theorem}[\cite{DJgood}, \cite{KS}]\label{T3}
Let $A$, $B$, $C$ and $D$ be pairwise amicable $\{-1,1\}$-matrices of order $p$, subject to \eqref{sup}. Let $X:=(A+B)/2$, $Y:=(A-B)/2$, $V:=(C+D)/2$, and $W:=(C-D)/2$. Then
\[K=\left[\begin{array}{cc}
X+\mathbf{i}Y & V+\mathbf{i}W\\
-V+\mathbf{i}W & X-\mathbf{i}Y\end{array}\right]\]
is a complex Hadamard matrix of order $2p$.
\end{theorem}
In order to construct a complex Hadamard matrix of order $94$, it is natural to try to find symmetric circulant matrices of order $47$ subject to \eqref{sup}, and combine them with Theorem~\ref{T3}. Unfortunately, such matrices do not exist \cite{WILKHAR}. Therefore, we modify the array $K$ of Theorem~\ref{T3} in order to accommodate a different class of matrices. The main idea is to somehow incorporate the matrix $R$ in $K$, similarly as shown in Theorem~\ref{T2}. First we need the following preliminary result.
\begin{lemma}\label{L1}
Let $X$ and $Y$ be complex matrices of order $p$, satisfying:
\begin{enumerate}
\item[(i)] $x_{ij},y_{ij}\in\{-1,-\mathbf{i},1,\mathbf{i}\}$ for every $i,j\in\{1,\dots,p\}$,
\item[(ii)] $XY=YX$,
\item[(iii)] $XX^\ast+YY^\ast =2pI_p$,
\item[(iv)] $X^\ast X+Y^\ast Y =2pI_p$.
\end{enumerate}
Then the matrix
\[Z=\left[\begin{array}{cc}
X & Y\\
Y^\ast & -X^\ast\end{array}\right]
\]
is a complex Hadamard matrix of order $2p$.
\end{lemma}
We remark that if $X$ and $Y$ are normal matrices in Lemma~\ref{L1}, then conditions (iii) and (iv) are equivalent.
\begin{proof}
By (i) the elements of $Z$ belong to $\{-1,-\mathbf{i},1,\mathbf{i}\}$. By (ii), the first block of rows and the second block of rows are pairwise complex orthogonal. Finally, by (iii) (resp.~(iv)), the first (resp.~second) block of rows are pairwise complex orthogonal.
\end{proof}
To the best of our knowledge, the following is a new construction.
\begin{theorem}\label{T4}
Let $A$, $B$, $C$ and $D$ be circulant $\{-1,1\}$-matrices of order $p$, subject to \eqref{sup}, such that $A$ and $B$ are symmetric. Let $R$ be the back-diagonal matrix. Then
\[\left[\begin{array}{cc}
(A+B)/2+\mathbf{i}(A-B)/2 & (C+DR)/2+\mathbf{i}(C-DR)/2\\
(C^T+DR)/2-\mathbf{i}(C^T-DR)/2 & -(A+B)/2+\mathbf{i}(A-B)/2
\end{array}\right]\]
is a complex Hadamard matrix of order $2p$.
\end{theorem}
\begin{proof}
We verify the conditions shown in Lemma~\ref{L1}. Let $X:=(A+B)/2+\mathbf{i}(A-B)/2$ and $Y:=(C+DR)/2+\mathbf{i}(C-DR)/2$. Condition (i) is clearly met. The matrix $X$ is symmetric and circulant, hence commutes with $C$, $D$, and $R$. In particular, condition (ii) is met. The matrix $X$ is circulant, hence normal. The matrix $Y$ is also normal by Lemma~\ref{L3}, since $C$ is circulant, $DR$ is real symmetric, and $C^T$ and $(DR)^T$ are amicable. Indeed, $C^T(DR)=DC^TR=DRC=(DR)(C^T)^\ast$. Thus, condition (iii) and (iv) are equivalent. Furthermore, $A$ and $B$ are self-adjoint, circulant, hence amicable by Lemma~\ref{L2} (i). The matrices $C$ and $DR$ are amicable by Lemma~\ref{L2} (ii). Thus, we may use Lemma~\ref{L3} to get $XX^\ast+YY^\ast=(AA^T+BB^T+CC^T+DD^T)/2$, which equals to $2pI_p$ by \eqref{sup}. In particular, condition (iii) is satisfied, and we are done.
\end{proof}
Therefore, in order to construct a complex Hadamard matrix of order $94$, it would be enough to find four, circulant $\{-1,1\}$-matrices of order $47$ satisfying \eqref{sup}, such that at least two of them are symmetric. Since we were unable to locate such matrices in the existing literature, we implemented a computer-aided search algorithm and constructed the required matrices directly.
\section{The search for circulant matrices of order $47$}\label{sect3}
The search outlined here is very similar to those described in detail in \cite{DK}, \cite{WILKHAR}, and \cite{MDK}. We recall the notion of symmetry types \cite{DJBIG} of a quadruple of circulant matrices $(A,B,C,D)$ satisfying \eqref{sup}. If all four matrices are symmetric, then the symmetry type is (ssss). This is the case of Williamson matrices \cite{WILKHAR}. If one of the matrices is skew, but the others are symmetric, then the symmetry type is (ksss). This is the case of good matrices \cite{DJgood}. When two of the matrices are skew and two are symmetric, then we have the symmetry type (kkss), the case of $G$-matrices \cite{SP}. Finally, when three matrices are skew, and only one of them is symmetric, is the case of best matrices with symmetry type (kkks). If a block is neither symmetric, nor skew, then (the lack of) its symmetry type is denoted by x. 

From now on we are focusing on matrices where (at least) two of the block are symmetric. It is known, see \cite[Table~1]{DJBIG} and \cite[Table~1]{DK}, that no Williamson matrices, good matrices, or $G$-matrices of order $p=47$ exist. In \cite[Section~5.4]{DJBIG} an example with (ksxx) symmetry is presented. Further, in \cite{DJsym} examples with symmetry type (sxxx) are shown. However, none of these seems to be compatible with our array in Theorem~\ref{T4}. For these reasons, we initiate a search for matrices with symmetry type (ssxx). As a preliminary, we collect the necessary tools.

\begin{lemma}
Let $A$, $B$, $C$, $D$ be symmetric circulant $\{-1,1\}$-matrices of odd order $p$ with row sums $\sigma_A$, $\sigma_B$, $\sigma_C$ and $\sigma_D$, respectively, such that \eqref{sup} holds. Then
\begin{equation}\label{sos}
\sigma_A^2+\sigma_B^2+\sigma_C^2+\sigma_D^2=4p.
\end{equation}
\end{lemma}
\begin{proof}
Multiply \eqref{sup} with the all $1$ matrix $J$ from the left and right, respectively.
\end{proof}
\begin{lemma}
Let $A$ be a symmetric, circulant matrix of odd order $p$ with first row $a=[a_0,\dots,a_{p-1}]$ which appears as a block of a Hadamard matrix of order $4p$. Then for every $k\in\{1,\dots,(p-1)/2\}$, we have
\begin{equation}\label{bound}
\left|a_0+2\sum_{j=1}^{(p-1)/2}a_j\cos(2\pi k j/p)\right|\leq 2\sqrt{p}.
\end{equation}
\end{lemma}
\begin{proof}
Indeed, the quantity on the left-hand side is the absolute value of an eigenvalue of $A$, which cannot be larger than the square root of the size of the Hadamard matrix in which the block $A$ is embedded.
\end{proof}
For $k\geq 1$ let $\mathcal{P}_k$ denote the $k$th prime number. Let $\mathcal{I}(a)\in\mathbb{Z}^{(p-1)/2}$ denote the vector containing the \emph{periodic autocorrelation coefficients} of $a=[a_0,\dots,a_{p-1}]$, i.e., coordinate $s\in\{1,2,\dots,(p-1)/2\}$ of $\mathcal{I}(a)$ is 
\[\mathcal{I}_s(a):=\sum_{k=0}^{p-1}a_ka_{k+s},\]
where indices on the right-hand side are taken modulo $p$. We define a hash function
\[h\colon \mathbb{Z}^{(p-1)/2}\to \mathbb{Z},\qquad  h(x):=\sum_{s=1}^{(p-1)/2}x_s\cdot \mathcal{P}_{500s},\]
and denote by $\|x\|:=\sum_{s=1}^{(p-1)/2}x_s^2$ the norm of $x$. Finally, let $\mathcal{B}\in\mathbb{Z}^+$ be an experimentally determined bound which controls the size of the search space, and in turn the required computer memory.

Recall that $R$ is the back-diagonal matrix. The following is essentially the search algorithm we implemented first in Mathematica\footnote{Wolfram Research, Inc., Mathematica, Version 14.3, Champaign, IL (2025).} for $p\leq 39$, and later in C++, where working with data structures supporting fast binary search is more convenient. To further improve efficiency, we precomputed the required prime numbers and the cosines (up to $15$ significant digits), and stored them in a lookup table.
\begin{enumerate}
\item[(A1)] Choose a positive odd integer $p\leq 47$, and a bound $\mathcal{B}\approx 1200$. Choose four odd integers $\sigma_A\equiv \sigma_B\equiv p\ (\mathrm{mod}\ 4)$, $\sigma_C>0$, $\sigma_D>0$, such that equation \eqref{sos} holds.
\item[(A2)] Generate all $\{-1,1\}$-vectors $x$ of length $(p-1)/2$ with row sum $(\sigma_A-1)/2$, and create all vectors $a=[1,x,xR]$, and analogously the vectors $b$.
\item[(A3)] For all those vectors created in (A2), test whether they satisfy each of the bounds \eqref{bound}. If so, then precompute the autocorrelation vectors $\mathcal{I}(a)$ and $\mathcal{I}(b)$.
\item[(A4)] Calculate all possible distinct sum of autocorrelation vectors $\Sigma:=\mathcal{I}(a)+\mathcal{I}(b)$, and if $\|\Sigma\|\leq \mathcal{B}$, then store $\Sigma$ at index $h(\Sigma)$ in a set.
\item[(A5)] Repeatedly choose two random vectors $c$ and $d$ of length $p$ with row sums $\sigma_C$ and $\sigma_D$, respectively, calculate $\Delta:=-\mathcal{I}(c)-\mathcal{I}(d)$, and check whether the set at index $h(\Delta)$ contains the vector $\Delta$. If so, save $c$, $d$, and $\Delta$, and proceed to (A6).
\item[(A6)] Recover some vectors $a$ and $b$, such that $\mathcal{I}(a)+\mathcal{I}(b)=\Delta$, and output $(a,b,c,d)$.
\end{enumerate}
For $p=47$ we found several examples with the required symmetry type (ssxx), but we only describe two of them. The first has $\sigma_A=3$, $\sigma_B=\sigma_C=7$, and $\sigma_D=9$, and was found by setting $\mathcal{B}=1000$. This resulted in $784031$ unique hash values, and a lookup table containing $136744706$ entries, with at most $445$ of them indexed by the same hash value. After roughly $8\cdot 10^9$ trials within a little over a day, we found the following sequences described in Example~\ref{ex1}. The computation was single-threaded, and required about 20GBs of memory.
\begin{example}\label{ex1}
Replacing the $0$s with $-1$s in the matrix
\[\left[\arraycolsep0pt
\begin{array}{ccccccccccccccccccccccccccccccccccccccccccccccc}
 1 & 1 & 1 & 1 & 0 & 1 & 0 & 1 & 1 & 0 & 0 & 0 & 1 & 0 & 1 & 1 & 0 & 0 & 0 & 0 & 0 & 1 & 1 & 1 & 1 & 1 & 1 & 0 & 0 & 0 & 0 & 0 & 1 & 1 & 0 & 1 & 0 & 0 & 0 & 1 & 1 & 0 & 1 & 0 & 1 & 1 & 1 \\
 1 & 1 & 0 & 1 & 0 & 1 & 1 & 1 & 0 & 1 & 0 & 1 & 1 & 0 & 0 & 0 & 1 & 1 & 0 & 0 & 1 & 0 & 1 & 1 & 1 & 1 & 0 & 1 & 0 & 0 & 1 & 1 & 0 & 0 & 0 & 1 & 1 & 0 & 1 & 0 & 1 & 1 & 1 & 0 & 1 & 0 & 1 \\
 1 & 1 & 1 & 1 & 1 & 1 & 1 & 1 & 1 & 0 & 0 & 0 & 1 & 1 & 0 & 1 & 0 & 1 & 1 & 0 & 0 & 1 & 0 & 0 & 0 & 0 & 0 & 1 & 0 & 0 & 1 & 1 & 0 & 1 & 1 & 1 & 1 & 0 & 1 & 1 & 1 & 0 & 0 & 1 & 0 & 1 & 0 \\
 1 & 1 & 1 & 1 & 1 & 1 & 0 & 0 & 1 & 1 & 1 & 1 & 0 & 1 & 1 & 0 & 1 & 0 & 1 & 0 & 0 & 1 & 1 & 0 & 1 & 1 & 0 & 0 & 0 & 1 & 0 & 1 & 1 & 0 & 1 & 1 & 1 & 0 & 0 & 0 & 1 & 1 & 0 & 1 & 1 & 0 & 0 \\
\end{array}
\right],\]
and cyclically shifting its rows $a$, $b$, $c$, and $d$, we get four circulant matrices $A=A^T$, $B=B^T$, $C$ and $D$ of order $47$ satisfying equation \eqref{sup}.
One may calculate that
\begin{align*}
\Sigma&:=\mathcal{I}(a)+\mathcal{I}(b)\\
&=[-2, 6, -10, -6, -6, 6, 2, -2, -10, 6, 6, -6, -6, -2, -6, -6, 2, 2, 2, -2, -2, 2, 14]
\end{align*}
with $\|\Sigma\|=796$. The sum of the autocorrelation coefficients peaks at $14$.
\end{example}
In particular, the proof of Theorem~\ref{T1} follows by combining the circulant matrices in Example~\ref{ex1} with the array shown in Theorem~\ref{T4}.

The second example has $\sigma_A=-1$, $\sigma_B=-5$, and $\sigma_C=\sigma_D=9$, and we found it by setting $\mathcal{B}=1200$.
\begin{example}\label{ex2}
Replacing the $0$s with $-1$s in the matrix
\[\left[\arraycolsep0pt
\begin{array}{ccccccccccccccccccccccccccccccccccccccccccccccc}
 1 & 0 & 0 & 1 & 1 & 0 & 0 & 0 & 0 & 1 & 1 & 1 & 0 & 0 & 0 & 1 & 0 & 0 & 0 & 1 & 1 & 1 & 1 & 1 & 1 & 1 & 1 & 1 & 1 & 0 & 0 & 0 & 1 & 0 & 0 & 0 & 1 & 1 & 1 & 0 & 0 & 0 & 0 & 1 & 1 & 0 & 0 \\
 1 & 0 & 0 & 0 & 1 & 1 & 0 & 0 & 1 & 0 & 1 & 0 & 0 & 0 & 1 & 0 & 1 & 1 & 0 & 0 & 1 & 0 & 1 & 1 & 1 & 1 & 0 & 1 & 0 & 0 & 1 & 1 & 0 & 1 & 0 & 0 & 0 & 1 & 0 & 1 & 0 & 0 & 1 & 1 & 0 & 0 & 0 \\
 1 & 1 & 1 & 1 & 1 & 1 & 1 & 0 & 0 & 0 & 0 & 1 & 0 & 1 & 1 & 1 & 1 & 0 & 0 & 1 & 1 & 0 & 1 & 0 & 1 & 0 & 1 & 1 & 0 & 0 & 0 & 1 & 0 & 0 & 1 & 1 & 0 & 1 & 0 & 1 & 1 & 1 & 1 & 0 & 1 & 1 & 0 \\
 1 & 1 & 1 & 1 & 1 & 1 & 0 & 1 & 0 & 1 & 1 & 1 & 1 & 0 & 1 & 0 & 1 & 1 & 0 & 1 & 1 & 1 & 0 & 1 & 0 & 0 & 1 & 0 & 1 & 1 & 0 & 1 & 1 & 1 & 1 & 0 & 1 & 0 & 0 & 0 & 1 & 0 & 0 & 1 & 0 & 0 & 0 \\
\end{array}
\right],\]
and cyclically shifting its rows $a$, $b$, $c$, and $d$, we get four circulant matrices $A=A^T$, $B=B^T$, $C$ and $D$ of order $47$ satisfying equation \eqref{sup}. In this case, we have
\begin{align*}
&\Sigma:=\mathcal{I}(a)+\mathcal{I}(b)\\
&=[6, -10, -10, 6, -6, 18, -10, -2, -6, 2, -10, 6, 2, -2, 2, -6, 2, -6, 2, -2, -10, 2, -2]
\end{align*}
with $\|\Sigma\|=1116$. Here the sum of the autocorrelation coefficients peaks at $18$.
\end{example}
The results of this paper may be extended in two ways: on the one hand, the array shown in Theorem~\ref{T4} might be further modified in order to accommodate only a single symmetric matrix (such as those found in \cite{DJsym}), or matrices meeting various other properties. On the other hand, it is plausible that access to supercomputing resources combined with more sophisticated search algorithms will eventually progress beyond the case $p=47$.

\appendix
\section*{Appendix (Online version only)}
Here we provide the four sequences from Example~\ref{ex1} in machine readable form.
\begin{verbatim}
{{1, 1, 1, 1,-1, 1,-1, 1, 1,-1,-1,-1, 1,-1, 1,
  1,-1,-1,-1,-1,-1, 1, 1, 1, 1, 1, 1,-1,-1,-1,-1,
 -1, 1, 1,-1, 1,-1,-1,-1, 1, 1,-1, 1,-1, 1, 1, 1},
 {1, 1,-1, 1,-1, 1, 1, 1,-1, 1,-1, 1, 1,-1,-1,
 -1, 1, 1,-1,-1, 1,-1, 1, 1, 1, 1,-1, 1,-1,-1, 1,
  1,-1,-1,-1, 1, 1,-1, 1,-1, 1, 1, 1,-1, 1,-1, 1},
 {1, 1, 1, 1, 1, 1, 1, 1, 1,-1,-1,-1, 1, 1,-1,
  1,-1, 1, 1,-1,-1, 1,-1,-1,-1,-1,-1, 1,-1,-1, 1,
  1,-1, 1, 1, 1, 1,-1, 1, 1, 1,-1,-1, 1,-1, 1,-1},
 {1, 1, 1, 1, 1, 1,-1,-1, 1, 1, 1, 1,-1, 1, 1,
 -1, 1,-1, 1,-1,-1, 1, 1,-1, 1, 1,-1,-1,-1, 1,-1,
  1, 1,-1, 1, 1, 1,-1,-1,-1, 1, 1,-1, 1, 1,-1,-1}}
\end{verbatim}

Here we provide the four sequences from Example~\ref{ex2} in machine readable form.
\begin{verbatim}
{{1,-1,-1, 1, 1,-1,-1,-1,-1, 1, 1, 1,-1,-1,-1,
  1,-1,-1,-1, 1, 1, 1, 1, 1, 1, 1, 1, 1, 1,-1,-1,
 -1, 1,-1,-1,-1, 1, 1, 1,-1,-1,-1,-1, 1, 1,-1,-1},
 {1,-1,-1,-1, 1, 1,-1,-1, 1,-1, 1,-1,-1,-1, 1,
 -1, 1, 1,-1,-1, 1,-1, 1, 1, 1, 1,-1, 1,-1,-1, 1,
  1,-1, 1,-1,-1,-1, 1,-1, 1,-1,-1, 1, 1,-1,-1,-1},
 {1, 1, 1, 1, 1, 1, 1,-1,-1,-1,-1, 1,-1, 1, 1,
  1, 1,-1,-1, 1, 1,-1, 1,-1, 1,-1, 1, 1,-1,-1,-1,
  1,-1,-1, 1, 1,-1, 1,-1, 1, 1, 1, 1,-1, 1, 1,-1},
 {1, 1, 1, 1, 1, 1,-1, 1,-1, 1, 1, 1, 1,-1, 1,
 -1, 1, 1,-1, 1, 1, 1,-1, 1,-1,-1, 1,-1, 1, 1,-1,
  1, 1, 1, 1,-1, 1,-1,-1,-1, 1,-1,-1, 1,-1,-1,-1}}
\end{verbatim}
\end{document}